\documentclass[12pt]{amsart} 
\usepackage{amssymb,amsmath,amsthm,amscd}
\usepackage{amsfonts}
\usepackage{mathtools}

\allowdisplaybreaks

\numberwithin{equation}{section}

\newtheorem{theorem}{Theorem}[section]
\newtheorem{lemma}[theorem]{Lemma}

\newtheorem{main}{Theorem}

\newtheorem*{main*}{Theorem}

\theoremstyle{remark}
\newtheorem{remark}[equation]{Remark}
\newtheorem*{remark*}{Remark}

\newcommand{\Ric}{\mathop{\mathrm{Ric}}\nolimits}
\newcommand{\Ad}{\mathop{\mathrm{Ad}}\nolimits}

\newcommand{\ecal}{\mathcal{E}}

\newcommand{\mcal}{\mathcal{M}}

\def\calM{\mcal}  
  
  \def\calE{\ecal}

\def\a{\alpha}

\def\max{{\operatorname{max}}}

\def\tr{\hbox{\rm tr}}

\def\Ad{\hbox{\rm Ad}}

\def\K{K\"ahler } \def\Kno{K\"ahler}
\def\KE{K\"ahler--Einstein }

\def\Ric{\hbox{\rm Ric}\,}

\def\h#1{\hbox{#1}}

\def\strutdepth{\dp\strutbox}
\def\specialstar{\vtop to \strutdepth{
    \baselineskip\strutdepth
    \vss\llap{$\star$\ \ \ \ \ \ \ \ \  }\null}}
\def\marginalstar{\strut\vadjust{\kern-\strutdepth\specialstar}}

\def\marginal#1{\strut\vadjust{\kern-\strutdepth
    {\vtop to \strutdepth{
    \baselineskip\strutdepth
    \vss\llap{{ \small #1 }}\null}
    }}
    }

\def\text{\textstyle}

\def\q{\quad} \def\qq{\qquad}

\def\ra{\rightarrow}

\newcommand{\PP}{{\mathbb P}} 
 \newcommand{\CC}{{\mathbb C}}
 \newcommand{\NN}{{\mathbb N}}

\def\la{\lambda}
\def\beq{\begin{equation}}
\def\eeq{\end{equation}}
\def\bpf{\begin{proof}}
\def\epf{\end{proof}}

\def\eaeq{\end{aligned}}
\def\baeq{\begin{aligned}}

\def\mf{\mathfrak}

\def\SO{\mathrm{SO}}
\def\SU{\mathrm{SU}}
\def\Sp{\mathrm{Sp}}

\def\U{\mathrm{U}}
\def\G2{\mathrm{G}_2}

\newcommand{\Sph}{\mathbb{S}}

\newcommand{\CP}{\mathbb{C\mkern1mu P}}
\newcommand{\HP}{\mathbb{H\mkern1mu P}}

\begin{document}

\title{On the Ricci iteration for homogeneous metrics on spheres and projective spaces}

\author{Timothy Buttsworth}
\address{The University of Queensland} 
\email{timothy.buttsworth@uq.net.au}
\author{Artem Pulemotov}
\address{The University of Queensland} 
\email{a.pulemotov@uq.edu.au}
\author{Yanir A. Rubinstein}
\address{University of Maryland} 
\email{yanir@umd.edu}
\author{Wolfgang Ziller}
\address{University of Pennsylvania} 
\email{wziller@math.upenn.edu}

\renewcommand\rightmark{}
\renewcommand\leftmark{}


\begin{abstract} We study the Ricci iteration for homogeneous metrics on spheres and complex projective spaces. Such metrics can be described in terms of modifying the canonical metric on the fibers of a Hopf fibration. When the fibers of the Hopf fibration are circles or spheres of dimension 2 or~7, we observe that the Ricci iteration as well as all ancient Ricci iterations can be completely described using known results. The remaining and most challenging case is when the fibers are spheres of dimension~3. On the 3-sphere itself, using a result of Hamilton on the prescribed Ricci curvature equation, we establish existence and convergence of the Ricci iteration and confirm in this setting a conjecture on the relationship between ancient Ricci iterations and ancient solutions to the Ricci flow. In higher dimensions we obtain sufficient conditions for the solvability of the prescribed Ricci curvature equation as well as partial results on the behavior of the Ricci iteration.
\end{abstract}

\maketitle

\def\lb{\label}


\section{Introduction and main results}
\label{IntroSec}

Let $(M,g_1)$ be a smooth Riemannian manifold.
A Ricci iteration is a sequence of metrics $g_i$ on $M$ satisfying
\begin{align*}
\label{RIEq}
\Ric g_{i+1}=g_i,\quad  i\in\NN,
\end{align*}
where $\Ric g_{i+1}$ denotes the Ricci curvature of $g_{i+1}$.
This concept was introduced by 
the third-named author \cite{R07,R08} as a discretization of the Ricci flow; see the survey~\cite[\S6.5]{R14} and references therein. In order to study a Ricci iteration, first one has to  understand when the   prescribed Ricci curvature equation has a solution. This is an old subject;
see, e.g., \cite[Chapter~5]{AB87},~\cite{TB16,MGAP17} and  references therein.

If $(M,g_1)$ is \Kno, a Ricci iteration exists if and only if
a positive multiple of $g_1$ represents the first Chern class, and the iteration
then converges modulo diffeomorphisms to a \KE metric
whenever one
exists~\cite{DR}.
In the non-\K case the Ricci iteration was studied only recently 
for a limited class of
homogeneous spaces \cite{PR}.
The purpose of this article is to add to these results in the non-\K case
by studying the Ricci iteration
for homogeneous metrics on Hopf
fibrations.
Indeed, such spaces have up to four
summands in their isotropy representations with possibly equivalent
isotropy summands, while most of the analysis of \cite{PR} pertains
to two inequivalent summands.

When studying the Ricci flow, following Hamilton \cite[\S19]{Ham1995}, it is important to understand the behavior of ancient solutions, i.e., solutions that can be extended indefinitely backward in time. These solutions are the prototype for singularity models for the Ricci flow and have been crucial, for example, in Perelman's work
\cite{Perelman}.
In our recent work, two of us proposed the
following discrete analogue of an ancient Ricci flow \cite[\S1]{PR}: an ancient Ricci iteration
is a sequence of Riemannian metrics
$g_i$ on $M$ such that
\begin{equation*}
\begin{aligned}
\label{RRIEq}
g_{i-1}=\Ric g_{i}, \qq i=1,0,-1,-2,\ldots.
\end{aligned}
\end{equation*}
The idea is that
ancient iterations should ``detect" ancient flows, i.e.,
the latter should exist if the former exist
\cite[Conjecture 2.5]{PR}.
It is rare that a Ricci iteration exists since the
prescribed curvature problem can often be
obstructed. It seems even rarer that
an ancient Ricci iteration exists since $\Ric\cdots\Ric g_{1}$ must
always be positive definite. If this is the case, we say that $g_1$ admits an ancient Ricci iteration. A basic question then is to study convergence and the behavior of the limit
since by the aforementioned conjecture this should
be helpful for studying ancient Ricci flows.

In this article we verify this conjecture for all homogeneous spheres and complex projective spaces. In the case of homogeneous metrics on $\Sph^3$, the following result completely describes both Ricci iterations and ancient Ricci iterations, and in particular confirms~\cite[Conjecture 2.5]{PR} in this setting as well. Notice that, given $T$, we can only hope to solve $\Ric g=cT$ up to a unique constant~$c$ (see Theorem~\ref{SO3PRC}). Thus we say that $g_i$ is a Ricci iteration starting from $cg_0$ if $\Ric g_1=cg_0$ and $\Ric g_{i+1}=g_i$ for $i\ge1$.

\begin{main}
\label{Main} Let $g_0$ be a left-invariant metric on $\SU(2)$.
\begin{itemize}
\item[(a)]
There exists a unique
Ricci iteration starting from $cg_0$ for some~$c>0$, and it converges  to a round metric.
\item[(b)] The only left-invariant metrics  which admit an ancient Ricci iteration are the Berger metrics $g^\lambda$ with $\lambda\in (0,1]$.
Unless $\la=1$, the sequence $g_i$ with $g_1=g^\lambda$ collapses
in the Gromov--Hausdorff topology, as $i\to - \infty$, to a round metric on $\Sph^2$ by shrinking the length of the Hopf fibers.
\end{itemize}
\end{main}
Recall that we have the Hopf fibration $\Sph^1 \xhookrightarrow{}  \Sph^3(1)\to \Sph^2(1/2)$ and that a Berger metric $g^\lambda$ 
is obtained by changing the length of the fibers to be equal to $2\pi\lambda$.
During the collapse the length  of the fibers goes monotonically to $0$, and hence the metric converges (in the Gromov--Hausdorff topology) to a metric on the base. We point out that one has the same behavior for ancient solutions of the Ricci flow of left-invariant metrics on $\Sph^3$ \cite{BKN}.

Homogeneous metrics on spheres and complex projective spaces were 
classified in \cite{WZ82}. Apart from the round sphere
and the Fubini--Study metric on complex projective spaces
in any real/complex dimension, these metrics can be described geometrically in terms of the Hopf fibrations:
\begin{align*}
\Sph^1 \xhookrightarrow{} \Sph^{2n+1}\to \CP^n,\qquad
\Sph^3&\xhookrightarrow{} \Sph^{4n+3}\to \HP^n,\qquad \Sph^7\xhookrightarrow{} \Sph^{15}\to\Sph^8,\\
\CP^1&\xhookrightarrow{} \CP^{2n+1}\to \HP^{n}.
\end{align*}
By scaling the canonical metric on the total space by a
constant $t$ in the direction of the fibers, and a constant $s$ perpendicular to it, we obtain the homogeneous metrics~$g_{t,s}$.
The only remaining homogeneous metrics are given  by
the family $g_{(t_1,t_2,t_3,s)}$ on $\Sph^{4n+3}$ where we modify the
round sphere metric with an arbitrary left invariant metric on the
3-sphere fiber. Up to isometry, one can assume the metric is
diagonal with respect to a fixed basis and $t_i$ are the lengths
squared of the basis vectors. For the metrics $g_{t,s}$ on $\Sph^{2n+1}$ and $\Sph^{15}$, as well as for the metrics
$g_{(t,t,t,s)}$ on $\Sph^{4n+3}$, the isotropy representation consists of two 
inequivalent irreducible
summands, a situation that was studied in detail in \cite{PR}  where Ricci iterations and ancient Ricci iterations were classified. We summarize the application to the metrics $g_{t,s}$ in Section~\ref{sec_12sum}.

The situation for the metrics $g_{(t_1,t_2,t_3,s)}$, which can be regarded as a generalization of the metrics in Theorem A, is more complicated. For the prescribed Ricci curvature problem we have the following sufficient condition.

\begin{main}\label{main2}
Assume that the homogeneous metric $T=g_{(a_1,a_2,a_3,b)}$ satisfies
$$
\frac b{a_i}< 2n+4,\quad i=1,2,3.
$$
Then there exists a homogeneous metric $g$  such that $\Ric  g=cT$ for some $c>0$.
\end{main}
In the special case where all $a_i$ are the same, this condition is necessary and sufficient; see Remark~\ref{rem_PRC_2par}.
 Among the metrics $g_{(x_1,x_2,x_3,s)}$ there exists a special subclass with $x_2=x_3$. These metrics are, in addition, invariant under the rotation of the plane corresponding to the $x_2,x_3$ variables. For these metrics, the sufficient condition in Theorem~\ref{main2} can be improved substantially; see Section~\ref{sec_4n3}.

For the question of the existence of Ricci iterations we obtain the following partial result.

\begin{main}\label{main3}
There exists a neighborhood $\mathcal O$ of a round metric $h_0$ in the space
of homogeneous metrics on $\Sph^{4n+3}$  such that for every $g_0\in\mathcal O$ a Ricci iteration starting with a multiple of $g_0$
exists and converges to a metric of constant curvature.
\end{main}

The organization is as follows. 
Section~\ref{sec_hm} describes the construction of homogeneous metrics on spheres and projective spaces in terms of Hopf fibrations and homogeneous spaces. Section~\ref{sec_12sum} summarizes the behavior of the Ricci iteration and ancient Ricci iteration for the two-parameter families $g_{t,s}$. Section~\ref{sec_S3} describes both behaviors for the left-invariant metrics on $\SU(2)\simeq \Sph^3$ and classifies the ancient Ricci iterations. The proofs of Theorems B and C appear in Sections~\ref{sec_4n3}--\ref{nearround}. Finally, we include in the Appendix a 
uniqueness result of independent interest concerning the prescribed Ricci curvature problem within the 4-parameter family of metrics on $\Sph^{4n+3}$.

\subsection*{Acknowledgements}
Parts of this work took place when Timothy Buttsworth was a visiting PhD student and Artem Pulemotov a Visiting Senior Lecturer at Cornell University. This research was supported by an Australian Government Research Training Program Scholarship (Timothy Buttsworth), ARC Discovery Early-Career Researcher Award DE150101548 (Timothy Buttsworth and Artem Pulemotov) and Discovery Project \linebreak DP180102185 (Artem Pulemotov), NSF grants DMS-1206284, DMS-1515703 (Yanir A.~Rubinstein) and DMS-1506148 (Wolfgang Ziller), and Sloan Research Fellowship (Yanir A.~Rubinstein).

\section{Homogeneous metrics on spheres and projective spaces}
\label{sec_hm}
\subsection{Homogeneous metrics on spheres}

Homogeneous metrics on spheres were classified by 
the fourth-named author \cite{WZ82}, and we now
review this classification. 
There are two ways to visualize such metrics. On the one hand, they can be described
via Hopf fibrations with fibers
$\Sph^1, \Sph^3, $ or $\Sph^7$. On the other hand,
they can be realized by
classical homogeneous space constructions. 

\def\ho{homogeneous }

Let us first give the former
description.
In even dimensions, the only
homogeneous metrics on the sphere are the round ones.
On odd-dimensional spheres, 
\ho metrics can be described geometrically in terms of Hopf fibrations,
$$
\h{\rm (i)}\quad\Sph^1\xhookrightarrow{} \Sph^{2n+1}\to \CP^n, 
\qquad
\h{\rm (ii)}\quad
\Sph^7\xhookrightarrow{} \Sph^{15}\to
\Sph^8, \qquad
\h{\rm (iii)}\quad
\Sph^3\xhookrightarrow{} \Sph^{4n+3}\to \HP^n.
$$ 
Scaling the round metric of curvature~1 on the total space by a constant $t>0$ in the direction of the fibers, and a constant $s>0$ in the horizontal direction, we obtain the metrics~$g_{t,s}$
in each of
the cases (i)--(iii).
The only remaining homogeneous metrics are given by
the family $g_{(t_1,t_2,t_3,s)}$ on $\Sph^{4n+3}$ for $n>0$ (see the beginning of Section~\ref{sec_4n3} for the details)
and the family of left-invariant metrics on~$\SU(2)$.

Let us now relate this to the homogeneous space description; see~\cite[\S2]{PR} and references therein for further details.
A homogeneous metric is a $G$-invariant Riemannian metric
on the homogeneous space
$$
M:=G/H,
$$
where $H$ is a closed subgroup 
of a Lie group $G$.
We assume that $G$ is compact and that $G$ and $H$ are connected. Let $\mathfrak g, \mathfrak h$ denote the Lie algebras of $G,H$, and let 
$Q$ be an $\Ad_G(G)$-invariant inner product on $\mathfrak g$.
The $Q$-orthogonal complement of $\mathfrak h$ in $\mathfrak g$ is an $\Ad_G(H)$-invariant subspace of $\mathfrak g$, denoted 
by $\mathfrak m$. Thus, 
$\mathfrak g=\mathfrak m\oplus \mathfrak h,
$
and
\begin{equation}\label{calMEq}
\begin{aligned}
\calM &:=\{\h{$G$-invariant Riemannian metrics on $M$}\}
\cr&
\cong\{\h{$\h{Ad}_G(H)$-invariant inner products on $\mathfrak m$}\}.
\end{aligned}
\end{equation}
Consider a $Q$-orthogonal $\Ad_G(H)$-invariant decomposition
\begin{align}\label{m_decomp}
\mathfrak m=\mathfrak m_1\oplus\cdots\oplus\mathfrak m_q
\end{align}
such that $\Ad_G(H)|_{\mathfrak m_i}$ is irreducible for each $i=1,\ldots,q$.
Ignoring the case $q=1$
(where $\dim\calM=1$)
the only homogeneous spheres
are as follows \cite[p. 352]{WZ82}:
\begin{equation*}
\begin{aligned}
q=2:& \quad \Sph^{2n+1}=
\SU(n+1)/\SU(n)=\U(n+1)/\U(n),
\cr 
& \quad \Sph^{4n+3}=
\Sp(n+1)\Sp(1)/\Sp(n)\Sp(1),
\cr
& \quad \Sph^{15}=\h{Spin}(9)/
\h{Spin}(7),
\cr
q=3:& \quad \Sph^{4n+3}=
\Sp(n+1)\U(1)/\Sp(n)\U(1),
\cr 
& \quad \Sph^3=\SU(2),
\cr 
q=4:& \quad
\Sph^{4n+3}=
\Sp(n+1)/\Sp(n).
\end{aligned}
\end{equation*}
On all of these, we assume $Q$ is induced by the round metric of curvature~1.

Let us explain how the two descriptions are tied.
Via \eqref{calMEq},
we can identify $\mathcal M$ with a subset of $\mathfrak m^*\otimes\mathfrak m^*$.
The Hopf fibrations can be written in the form
$$ K/H \ra G/H\ra G/K,$$
where $K$ is a subgroup of~$G$ containing~$H$.
The representation $\Ad_G(H)$ on $K/H$
is 1-, 3- or 7-dimensional 
if $q=2$, splits into three
1-dimensional subspaces
if $G=\SU(2)$ or $q=4$, and splits into
two irreducible subspaces of dimensions 1 and 2 if $G\ne\SU(2)$ and $q=3$.
Thus, say in the 
case of $q=2$,
\begin{align*}
g_{t,s}
=
t\pi_1^* Q
+
s\pi_2^* Q,
\end{align*}
where $\pi_i:\mathfrak m\ra \mathfrak m_i
$ denote the natural projections induced by~\eqref{m_decomp}.

Another useful observation is
that the $\Sp(n+1)\Sp(1)$- and $\Sp(n+1)\U(1)$-invariant metrics
on $\Sph^{4n+3}$ are 
actually all contained
in the family $\{g_{(t_1,t_2,t_3,s)}\}$ (see Section~\ref{sec_4n3} for the rigorous definition of this family). The $\Sp(n+1)\Sp(1)$-invariant ones are precisely the metrics
$\{g_{(t,t,t,s)}\}$,
while the $\Sp(n+1)\U(1)$-invariant ones are precisely the metrics
\begin{equation*}
\begin{aligned}
\{g_{(t,t,u,s)}\}
\cup \{g_{(t,u,t,s)}\}
\cup \{g_{(u,t,t,s)}\}.
\end{aligned}
\end{equation*}

\subsection{Homogeneous metrics on complex projective spaces }

Once again, we ignore the 
case $q=1$. One is then left with
the odd-dimensional 
complex projective spaces
$\CC\PP^{2n+1}$ that
can be described as
$\Sp(n+1)/\Sp(n)\U(1)$,
and in this case, $q=2$
\cite[p. 356]{WZ82}.
For similar reasons to those in the 
previous subsection, the 2-parameter space
of homogeneous
metrics
$
\{t\pi_1^* Q
+
s\pi_2^* Q\,:\, t,s>0
\}$, with an appropriate choice of $Q$,
coincides with the family $\{g_{t,s}\}$
obtained from the Hopf fibration
$\mathbb S^2\xhookrightarrow{} \CP^{2n+1}\to \HP^{n}$ 
described in Section~\ref{IntroSec}.


\section{Two isotropy components}
\lb{2compSec}
\label{sec_12sum}

In this section we briefly recall how some of our previous work handles the case of circle, 2-sphere and 7-sphere
fibers, as well as 3-sphere
fibers with additional symmetry.
Namely, the cases we consider here are:
\begin{equation}
\begin{aligned}
\label{s2casesEq}
\Sph^{2n+1} &=
\SU(n+1)/\SU(n)=\U(n+1)/\U(n),
\cr 
\Sph^{4n+3} &=
\Sp(n+1)\Sp(1)/\Sp(n)\Sp(1),
\cr
\Sph^{15} &=\h{Spin}(9)/
\h{Spin}(7),
\cr
\CC\PP^{2n+1} &=
\Sp(n+1)/\Sp(n)\U(1).
\end{aligned}
\end{equation}
Thus, $(\dim \mf m_1,\dim \mf m_2)\in\{(1,2n),(3,4n),(7,8),
(2,4n)\}$.
This means $\dim \mf m_1\not=\dim \mf m_2$, so
$\Ad_G(H)|_{\mathfrak m_1}$ 
is inequivalent to $\Ad_G(H)|_{\mathfrak m_2}$,
i.e.,
the main assumption of
\cite[Theorems 2.1, 2.4]{PR}
is satisfied.
The other assumption in
those theorems is that
$Q([X,Y],Z)\ne0$ 
for some $X\in\mathfrak m_1$ and $Y,Z\in\mathfrak m_2$.
As explained in
\cite[\S2]{PR}, this holds unless
all metrics in $\calM$
have the same Ricci curvature,
which is not
the case for the spaces
\eqref{s2casesEq} by the
curvature formulas
in \cite{WZ82}.
Finally, to formulate the conclusion of these theorems
we need one more piece of information on the spaces
\eqref{s2casesEq}, namely,
the classification of 
Einstein metrics on them.
Denote
\begin{equation*}
\begin{aligned}
\calE&:=\{\h{Einstein metrics in $\calM$}\},\\
\alpha_-&:=
\inf\big\{{t}/{s}\,:\,
T=
t\pi_1^* Q
+
s\pi_2^* Q
\in\mathcal E\big\},
\\
\alpha_+&:=
\sup\big\{{t}/{s}\,:\,
T=
t\pi_1^* Q
+
s\pi_2^* Q
\in\mathcal E\big\}.
\end{aligned}
\end{equation*}
Thus, 
$(\alpha_-,\alpha_+)\in
\{(1,1),(1/(2n+3),1),(3/11,1),(1/(n+1),1)\}$
\cite{WZ82}.
Given all this, we can now
completely describe the Ricci iteration
and ancient Ricci iterations on \eqref{s2casesEq}.

We start with the simplest case of
$\mathbb S^{2n+1}$ where
there is a unique Einstein metric. In this case $H$
is not maximal, but 
a central extension thereof is, with Lie algebra~$\mf h\oplus \mf u(1)$.
The following is a consequence of~\cite[Theorems 2.1 (ii-a), 2.4 (ii-a)]{PR}.

\begin{theorem}\label{2par_dV1}
Let $g_{t,s}$ be a homogeneous metric on $\Sph^{2n+1}$. Then:
\begin{itemize}
\item[(a)] There exists a unique Ricci iteration starting with $cg_{t,s}$ for some $c>0$, and it
smoothly converges to a round metric.
\item[(b)] There exists an ancient  Ricci iteration starting with $g_{t,s}$ if and only if $t\le s$. If $t< s$, this iteration converges in the Gromov--Hausdorff topology to a multiple of the Fubini--Study metric on
$\CC\PP^n$ by shrinking the
fibers to $0$, i.e., $t\to 0$.
\end{itemize}
\end{theorem}

In the remaining three cases,
$H$
is not maximal and nor 
does $\Ad_G(H)$ act trivially
on~$\mf m_1$,
and there are two Einstein
metrics. 
The following is a consequence of~\cite[Theorems 2.1 (ii-b), 2.4 (ii-b)]{PR}.

\begin{theorem}\label{dV3} Assume that $g_{t,s}$ is a homogeneous metric on
$\Sph^{4n+3}$
with fibers of dimension
$3$ (or $\Sph^{15}$
with fibers of dimension
$7$, or $\CC\PP^{2n+1}$). Then:
\begin{itemize}
\item[(a)] There exists a  Ricci iteration starting with $cg_{t,s}$ for some $c>0$ if and only if $t/s\ge 1/(2n+3)$
(or $t/s\ge 3/11$, or
$t/s\ge 1/(n+1)$). Such a Ricci iteration is unique. Unless $t/s = 1/(2n+3)$ 
(or $t/s= 3/11$, or
$t/s= 1/(n+1)$)
such an iteration converges towards a round metric (or a multiple of the standard metric in the case of $\CC\PP^{2n+1}$).
\item[(b)] There exists an ancient Ricci iteration starting with $g_{t,s}$ if and only if $t/s\le 1$. If 
$t/s<1$, this ancient iteration converges to the second Einstein metric in $\calE$.
\end{itemize}
\end{theorem}

These results show that the behavior of ancient Ricci iterations in the four
cases of this section is the same as for ancient solutions to the Ricci flow~\cite{BKN}.

\begin{remark}\label{rem_PRC_2}
One can give an alternative description of 
Theorems~\ref{2par_dV1}--\ref{dV3} that relies on the 
Hopf fibrations picture from \cite{WZ82} instead of the
languange of homogeneous spaces. 
For the case of spheres in these theorems, if $V$ and $H$ are the vertical and horizontal space of a 
Hopf fibration of dimensions $d_V$ and $d_H$, then 
 $g_{t,s}=t{\hat g}_{|V}+s{\hat g}_{|H}$, where $\hat g$ is the metric of curvature $1$ on $\Sph^N, N= d_V+d_H$, and 
for $u\in V, x\in H$,
\begin{align*}
\Ric g_{t,s}(u,u)&=(d_V - 1) + d_H{t^2}/{s^2},\\ 
\Ric g_{t,s}(x,x)&=d_H + 3d_V - 1 - 2d_Vt/s, \qquad
\Ric g_{t,s}(u,x)=0.
\end{align*}
One can then obtain Theorems 
\ref{2par_dV1}--\ref{dV3} directly from these formulas
by monotonicity arguments as in \cite[\S4.3--4.4]{PR}. 
Similarly, for 
the homogeneous metrics on $\CP^{2n+1}$ one can use the 
Hopf fibration
$\CP^1\xhookrightarrow{} \CP^{2n+1}\to \mathbb{HP}^{n}$, where now 
for $u\in V, x\in H$,
$$
\Ric g_{t,s}(u,u)=4n+8-4t/s, \quad
\Ric g_{t,s}(x,x)=4nt/s-4s/t,\quad
\Ric g_{t,s}(u,x)=0.
$$
\end{remark}

\begin{remark}\label{rem_PRC_2par}
In the setting of Theorem~\ref{dV3}, the condition for solving the equation  $\Ric h=cg_{t,s}$ is weaker than the condition for the existence of the Ricci iteration. In fact, a solution exists if and only if
 $t/s> 1/(2n+4)$ 
(or $t/s> 3/14$, or 
$t/s> 1/(n+2)$), as a computation starting from the formulas in Remark \ref{rem_PRC_2}
shows; cf.~\cite[Proposition~3.1]{AP16}.
This demonstrates that the behavior of the Ricci iteration can be different from the behavior of the Ricci flow \cite{MB14} since in some cases even the first iteration is not possible. Notice though that one also has solutions for some negative values of $s$, i.e., the prescribed Ricci tensor does not have to be positive definite.
\end{remark}

\section{Metrics on $\Sph^3$}\label{sec_S3}

In this section, we denote by
$$\calM=\calM(\Sph^3)=\calM(\SU(2))$$
 the set of left invariant metrics on $\Sph^3=\SU(2)$.
Given a metric $g\in\calM$, we can diagonalize $g$ with respect to a basis $\{e_1,e_2,e_3\}$ such that
 \begin{align}\label{Lie}
[e_i,e_{i+1}]=2e_{i+2}, \q i\in\{1,2,3\},
\end{align}
with indices mod 3; thus
\begin{align}\label{diag}
g(e_i,e_j)=x_i\delta_{ij}
\end{align}
for some $x_1,x_2,x_3>0$.  A computation shows that~\cite[\S6]{RH84}
\begin{align}\label{eq_Ricci}
\Ric g\,(e_i,e_j)
&=
\frac{2(x_i^2-(x_{i+1}-x_{i+2})^2)}{x_{i+1}x_{i+2}}\delta_{ij}
=:r_i\delta_{ij}.
\end{align}
Thus $\Ric g$ is diagonal with respect to the same basis.
The following result by Hamilton \cite[Theorem~6.1]{RH84} implies the existence of Ricci iterations on~$\SU(2)$.

\begin{theorem}[Hamilton]\label{SO3PRC}
For every metric  $T\in\calM(\Sph^3)$,
there exists a  left-invariant Riemannian metric $g$, unique up to scaling, such that $\Ric g=cT$ for some positive constant~$c$.
\end{theorem}

\begin{remark}
In addition to Hamilton's original proof,
the existence portion of Theorem~\ref{SO3PRC} can be proven by variational methods  based on~\cite[Lemma~2.1]{AP16}. In fact, one can show that the metric $g$ is the global maximum point of the scalar curvature functional on the set $\mathcal M_T(\SU(2))=\{h\in\mathcal M(\SU(2))\,|\,\tr_hT=1\}$
~\cite{AP18}.
\end{remark}

\subsection{Convergence of the Ricci iteration}
\label{sec_pf_conv}

We now prove Theorem~\ref{Main} (a). The proof  relies on a simple monotonicity lemma.
 To state it, suppose $g$ is a left-invariant metric on $\SU(2)$
 satisfying \eqref{diag}.
From \eqref{eq_Ricci} it follows that:

\begin{align}\label{eq_aux111}
\frac{r_i}{r_j}
=
\frac{x_{\{i,j\}^c}+(x_i-x_j)}
{x_{\{i,j\}^c}-(x_i-x_j)}
\cdot \frac{x_{i}}{x_j},
\end{align}
where $\{i,j\}^c:=\{1,2,3\}\setminus\{i,j\}$, provided $r_j\not=0$.
This shows:

\begin{lemma}\label{lem_monot}
Assume the Ricci curvature of the metric $g$ is nondegenerate.
If
$$
{x_i}/{x_j}\ge1
\q (<1),
$$
then
$$
{r_i}/{r_j}\ge{x_i}/{x_j}
\q
({r_i}/{r_j}<{x_i}/{x_j})
$$
for all $i,j\in\{1,2,3\}$.
\end{lemma}

 Given a left-invariant metric $T$ on $\SU(2)$, Theorem~\ref{SO3PRC} implies that there exists a left-invariant metric $g$, unique up to scaling, such that $\Ric g=cT$ for some $c>0$. This fact implies the existence and the uniqueness of the sequence~$\{g_i\}_{i\in\mathbb N}$. In fact, $g_i$ are all diagonal with respect to $\{e_1,e_2,e_3\}$. Thus:
\begin{align*}
g_i(e_k,
e_l)
=x_k^{(i)}\delta_{kl},
\q k\in\{1,2,3\}, \;
i\in\NN,
\end{align*}
with $x_k^{(i)}>0$. Set
 $$
\a_{kl}^{(i)}:=x_k^{(i)}/{x_l^{(i)}}.
 $$
Lemma~\ref{lem_monot} implies
that the sequence
$
\big\{\a_{kl}^{(i)}\big\}_{i\in\mathbb N}$ is monotone
for every $k,l$ and
that
$$
\min\{\a_{kl}^{(1)},1\}
\le\a_{kl}^{(i)}
\le
\max \{\a_{kl}^{(1)},1\}
$$
for all $k,l$ and all $i$.
Thus, this sequence converges to some $\alpha_{kl}>0$.
By~\eqref{eq_Ricci},
\begin{align}\lb{alphaklEq}
\a_{kl}^{(i)}
=
\frac{x_k^{(i)}}{x_l^{(i)}}
=
\frac{x_k^{(i+1)}}{x_l^{(i+1)}}
\cdot
\frac{x_{\{k,l\}^c}^{(i+1)}+x_k^{(i+1)}-x_l^{(i+1)}}
{x_{\{k,l\}^c}^{(i+1)}-x_k^{(i+1)}+x_l^{(i+1)}}
=
\a_{kl}^{(i+1)}
\frac{
\a_{\{k,l\}^ck}^{(i+1)}+1-\a_{lk}^{(i+1)}
}
{
\a_{\{k,l\}^ck}^{(i+1)}-1+\a_{lk}^{(i+1)}
},
\qquad i\in\mathbb N.
\end{align}
Passing to the limit,
\begin{align*}
\alpha_{kl}
\big(\a_{\{k,l\}^ck}-1+\a_{lk}
\big)
=\alpha_{kl}
\big(\a_{\{k,l\}^ck}+1-\a_{lk}\big),
\end{align*}
whence $\alpha_{lk}=1$. By~\eqref{eq_Ricci},
\begin{align}
\begin{split}
x_k^{(i)}
&=
2\frac{\big(x_k^{(i+1)}+x^{(i+1)}_{k+2}-x^{(i+1)}_{k+1}\big)\big(x^{(i+1)}_k+x^{(i+1)}_{k+1}-x^{(i+1)}_{k+2}\big)}{x^{(i+1)}_{k+1}x^{(i+1)}_{k+2}}
\lb{xkiEq}\cr
&=2\big(\a^{(i+1)}_{kk+1}+\a^{(i+1)}_{k+2k+1}-1\big)
\big(\a^{(i+1)}_{kk+2}+\a^{(i+1)}_{k+1k+2}-1\big).
\end{split}
\end{align}
Passing to the limit again,
\begin{align*}
\lim_{i\to\infty} x_k^{(i)}=2,
\end{align*}
so $\{g_i\}_{i\in\mathbb N}$ converges to a round metric on $\Sph^3$.

\subsection{Classification of ancient Ricci iterations}
\label{sec_pf_ancient}

Recall that a Berger metric is a metric as in \eqref{diag}, with respect to some  basis satisfying~\eqref{Lie}, and with $(x_1,x_2,x_3)=(\nu,2,2)$.  If we let $\mathcal B$ be the set of all such bases, then we denote by $g^\nu_D$ the Berger metric with respect to the basis $D\in\mathcal B$. Notice that one obtains a round metric for $\nu=2$.

In order to prove Theorem A (b), we can assume that the ancient Ricci iteration $g_{i-1}=\Ric g_{i}$ for $g_i\in\calM(\SU(2))$ has the property that all $g_i$ are diagonal with respect to a fixed basis $\{e_1,e_2,e_3\}$.

\begin{lemma}\label{Berger}
The Berger metric $g^\nu_D$ admits an ancient Ricci iteration if and only if
$\nu\in(0,2]$.
\end{lemma}

\begin{proof}
Let $g_1=g^\nu_D$. First, suppose $\nu\in(0,2)$. It suffices to show that
\begin{equation}
\begin{aligned}
\label{xkiposEq}
x_k^{(i)}>0 \h{\ \ for all $k\in\{1,2,3\}$ and $i\le 1$}.
\end{aligned}
\end{equation}
Assume by induction that
$0<x_1^{(j)}<2\le
x_2^{(j)}=x_3^{(j)}$
for all $j\in\{1,0,\ldots,i+1\}$.
We claim this holds also for $j=i$, which
then, of course, implies  \eqref{xkiposEq}.
Indeed, by~\eqref{eq_Ricci},
\begin{equation*}
\begin{aligned}
\label{xii1Eq}
x_1^{(i)}=2\big(x_1^{(i+1)}/x_2^{(i+1)}\big)^2,
\q
x_2^{(i)}=x_3^{(i)}
=
2\big(2x_2^{(i+1)}-x_1^{(i+1)}\big)/x_2^{(i+1)}.
\end{aligned}
\end{equation*}
Evidently then $x_1^{(i)}>0$ but also
by induction we see that $x_1^{(i)}<2<
x_2^{(i)}=x_3^{(i)}$.

Next, we assume that $\nu>2$ and  that $g^\nu_D$ admits an ancient Ricci iteration.
Then
the argument of the previous paragraph shows that
$0<x_2^{(i)}=x_3^{(i)}\le2<
x_1^{(i)}
$
for all $i\le 1$.
So, $\a^{(i)}_{23}=\a_{23}=1$. Also,
$\a^{(i)}_{21}=\a^{(i)}_{31}<1$, and
by Lemma~\ref{lem_monot} the sequences $\{\a^{(i)}_{21}\}_{i\in\mathbb N}$ and $\{\a^{(i)}_{31}\}_{i\in\mathbb N}$
are monotone decreasing. The limits of these sequences are in~$[0,1)$.
Suppose $\a_{21}>0$.
Passing to the limit in~\eqref{alphaklEq}, we obtain
$\a_{21}
=
\a_{21}
\frac{
\a_{32}+1-\a_{12}
}
{
\a_{32}-1+\a_{12}
},
$
which implies that $\a_{12}=1$. Hence $\a_{21}=1$, a contradiction. We conclude that
$\a_{21}=0$ and
$\lim_{i\ra-\infty}\a_{12}^{(i)}=\infty$. However, going back to
 \eqref{xii1Eq}, we get
 \begin{equation}
 \begin{aligned}
 \label{negEq}
 x_2^{(i)}=x_3^{(i)}
=
2(2x_2^{(i+1)}-x_1^{(i+1)})/x_2^{(i+1)}
=4-2\a^{(i+1)}_{12}.
\end{aligned}
 \end{equation}
This quantity is negative for $i$ sufficiently close to $-\infty$. Since $x_2^{(i)}$ is the component of a Riemannian metric, we obtain a contradiction.
\end{proof}

If $g_1=g^\nu_D$ with $\nu\in(0,2)$,
the arguments from the proof of Lemma~\ref{Berger} show that
$\a_{12}=0$. Similarly, $\a_{13}=0$.
Going back to~\eqref{xkiEq},
 $$
 x_1^{(i)}
 =2\big(\a^{(i+1)}_{12}+\a^{(i+1)}_{32}-1\big)
 \big(\a^{(i+1)}_{13}+\a^{(i+1)}_{23}-1\big)
 =2\a^{(i+1)}_{12}\a^{(i+1)}_{13}
$$
yields
$\lim_{i\ra-\infty} x_1^{(i)}
=0$. Going back to  \eqref{negEq}, we get
$$
\lim_{i\ra-\infty}
x_2^{(i)}
=
\lim_{i\ra-\infty}
x_3^{(i)}
=
4.
$$
Thus, the $\mathbb S^1$ fibers of the Hopf
fibration collapse, and $g_i$ converge in the Gromov--Hausdorff topology to a round metric on $\Sph^2$.

The next lemma completes the classification of homogeneous ancient Ricci iterations.

\begin{lemma}\label{lastLemma}
The metric $g\in\mathcal M$ admits an ancient Ricci iteration if and only if
$g=cg^\nu_D$ for some $c>0$, $\nu\in(0,2]$ and $D\in\mathcal B$.
\end{lemma}

\begin{proof}
Let $g$ be such that $g\not= cg^\nu_D$ for any $c,\nu>0$ and $D\in\mathcal B$.
Assume that there exists an ancient Ricci iteration starting with~$g=g_1$. Since $g\not= cg^\nu_D$, we can assume that
\begin{equation}
\begin{aligned}
\label{dataEq}
x_1^{(1)}<x_2^{(1)}<x_3^{(1)}.
\end{aligned}
\end{equation}

Lemma \ref{lem_monot} and \eqref{dataEq} imply that
$\a_{21}^{(i)}, \a_{31}^{(i)},$ and $ \a_{32}^{(i)}$
are all monotonically increasing.
 We claim that they all converge.
Indeed, by assumption, $x_k^{(i)}>0$ for all $i\le 1$,
so by~\eqref{eq_aux111},
$
x_{\{k,l\}^c}^{(i)}+x^{(i)}_k-x^{(i)}_l>0
$
for all $k,l$. Therefore,
$$
\a_{32}^{(i)}< \a_{12}^{(i)}+1.
$$
The right-hand side  is monotonically decreasing, while the left-hand side is monotonically increasing.
Thus, both sides converge, and since $\a_{32}>\a_{32}^{(1)}>1$,
we get $\a_{12}>0$. Also, $\a_{31}=\a_{32}/\a_{21}>0$.
However, passing to the limit
in  \eqref{alphaklEq}, we obtain
$\a_{21}
=
\a_{21}
\frac{
\a_{32}+1-\a_{12}
}
{
\a_{32}-1+\a_{12}
},
$
which gives $\a_{12}=1$. This is a contradiction, as
$\a_{12}<\a_{12}^{(1)}<1$.
\epf

\section{Four-parameter family of metrics}\label{sec_4n3}

We now discuss the homogeneous metrics on $\Sph^{4n+3}=\Sp(n+1)/\Sp(n)$. The gauge group $\Sp(1)=N(H)/H$ acts on the vertical space $V$ as $\SO(3)$ via the twofold cover $\Sp(1)\to \SO(3)$
 \cite[p. 353]{WZ82}. It thus acts transitively on the set of (oriented) bases orthonormal with respect to the metric on $V$ induced by the round metric $\hat g$ of curvature~1.
Hence, given any $\Sp(n+1)$-invariant metric $g$, we can assume that there exists a basis $\{e_1,e_2,e_3\}$ of $V$ satisfying~\eqref{Lie} in which $g$ is diagonal, i.e.,
$$
g(e_i,e_j)=x_i\delta_{ij},\quad g_{|H}=s\,\hat g\quad \textrm{and} \quad g(e_i,H)=0
$$
for some positive constants $x_1,x_2,x_3,s$.  We denote this metric by $g=g_{(x_1,x_2,x_3,s)}$.

The Ricci curvature of $g$ satisfies, for all $u\in H$,
\begin{align*}
\Ric g(e_i,e_j)
&=\bigg(4n\frac{x_i^2}{s^2}+2\frac{x_i^2-(x_{i+1}-x_{i+2})^2}{x_{i+1}x_{i+2}}\bigg)\delta_{ij},
\\ \Ric g(u,u)&=\Big(4n+8-2\frac{x_1+x_2+x_3}s\Big)g(u,u), \\
\Ric g (e_i,u)&=0,
\end{align*}
and is thus again diagonal with respect to the same basis.

We now study the question of prescribing the Ricci curvature. Let $T$ be a metric invariant under $\Sp(n+1)$. We want to solve $\Ric g=\kappa T$, $\kappa >0$, for a homogeneous metric $g$. Assuming $T$ is diagonal in some basis $\{e_1,e_2,e_3\}$ of $V$, we set
\begin{align}\label{T_4comp}
T(e_i,e_j)=T_i\delta_{ij},\quad T_{|H}=b\,\hat g.
\end{align}

We find the following sufficient condition:

\begin{theorem}\label{thm_PRC_4n3}
Assume the $\Sp(n+1)$-invariant metric $T$ on $\Sph^{4n+3}$ satisfies
\begin{align}\label{eq_suff_cond}
\frac b{T_i}<2n+4, \quad i=1,2,3.
\end{align}
Then there exists an $\Sp(n+1)$-invariant metric $g$  such that $\Ric g=\kappa T$ for some $\kappa>0$.
\end{theorem}

\begin{remark}
According to Remark~\ref{rem_PRC_2par}, the equation $\Ric g=\kappa g_{a,b}$ has a solution if and only if $b/a<2n+4$. The above theorem generalises the ``if'' part of this statement.
\end{remark}

\begin{proof} It is sufficient to prove the claim if $b=1$. We will prove the existence of a metric $g=g_{(x_1,x_2,x_3,s)}$, diagonal in the basis $\{e_1,e_2,e_3\}$, with Ricci curvature $\kappa T$. Since $\Ric$ is scale-invariant, we can assume that $s=1$.
Consider the following system of equations depending on a parameter $\lambda$:
\begin{align}\label{prce}
 c&=(4n+8)-2(x_1+x_2+x_3),\notag \\
 x_2x_3cT_1&=\lambda x_1^2x_2x_3+2(x_1^2-(x_2-x_3)^2),\notag \\
 x_1x_3cT_2&=\lambda x_2^2x_1x_3+2(x_2^2-(x_1-x_3)^2),  \\
 x_1x_2cT_3&=\lambda x_3^2x_2x_1+2(x_3^2-(x_1-x_2)^2).\notag
\end{align}
 If $\lambda=4n$, a solution to the system gives us the desired solution to $\Ric g=\kappa T$ with $\kappa =c$. By varying  $\lambda$  from $0$ to $4n$, we will show that \eqref{prce} has a solution for all $\lambda\in[0,4n]$ using degree theory.

Notice that if $\lambda=0$, the last three equations
\begin{align}\label{su2}
 x_2x_3cT_1&=2(x_1^2-(x_2-x_3)^2),\notag \\
 x_1x_3cT_2&=2(x_2^2-(x_1-x_3)^2), \\
 x_1x_2cT_3&=2(x_3^2-(x_1-x_2)^2) \notag
\end{align}
are the prescribed Ricci curvature equations for a left-invariant metric on $\SU(2)$. Recall that Hamilton's Theorem \ref{SO3PRC} says that we can solve equations \eqref{su2}, and that the solution $(x_1,x_2,x_3,c)$ is unique up to scaling of $(x_1,x_2,x_3)$. By examining his proof one easily sees that this solution depends differentiably on $T_i$.

In order to solve the system with $\lambda\ne0$ we first obtain the following bound:
\begin{lemma}\label{bound}
If $(x_1,x_2,x_3,c)$ solve the system \eqref{su2} and $1/T_i<2n+4$ for all $i$, then $c=c(T_1,T_2,T_3)<4n+8$.
\end{lemma}
\begin{proof}
By dividing each of the equations in \eqref{su2} by $x_1x_2x_3$ and adding, we obtain:
\begin{align*}
c\sum_{i=1}^3\frac{T_i}{x_i}&=4\sum_{i=1}^3\frac1{x_i} -2\frac{x_1}{x_2x_3}-2\frac{x_2}{x_1x_3}-2\frac{x_3}{x_1x_2}\\
&=4\sum_{i=1}^3\frac1{x_i}-\frac1{x_1}\Big(\frac{x_2}{x_3}+\frac{x_3}{x_2}\Big)-\frac1{x_2}\Big(\frac{x_1}{x_3}+\frac{x_3}{x_1}\Big)
-\frac1{x_3}\Big(\frac{x_1}{x_2}+\frac{x_2}{x_1}\Big)\\
&\le 4\sum_{i=1}^3\frac1{x_i}-2\sum_{i=1}^3\frac1{x_i}=2\sum_{i=1}^3\frac1{x_i}
\end{align*}
since $a+1/a\ge 2$ for $a>0$. This implies the claim.
\end{proof}

In order to apply degree theory, we now show that the set of solutions of  \eqref{prce} with   $(x_1,x_2,x_3,c)\in (0,\infty)^4$ and $\lambda\in [0,4n]$ is compact.

\begin{lemma}\label{APEPRCE}
 There exists a bounded convex open subset $\Omega$ of $(0,\infty)^{4}$ such that for  $\lambda\in [0,4n]$,
 any solution $(x_1,x_2,x_3,c)$ of \eqref{prce} lies in $\Omega$.
\end{lemma}
\begin{proof}
 Assume to the contrary that no such set exists.
 Then there is a sequence of $\lambda^{(i)}\in [0,4n]$ with a corresponding
 sequence $\big(x_1^{(i)},x_2^{(i)},x_3^{(i)},c^{(i)}\big)\in (0,\infty)^{4}$ of solutions to \eqref{prce} such that one of the variables goes to $0$ or $\infty$. For the remainder of the proof, we surpress reference to $i$ to simplify notation. The first equation in~\eqref{prce} shows that no variable can go to~$\infty$. We will consider two cases, first that $c\to 0$, and second, that at least one of $x_1,x_2$ or $x_3$
 goes to $0$ and $c$ is does not converge to $0$. We show that we get a contradiction in both cases.
\vspace{5pt}

 \textbf{First} \textbf{Case}. If $c\to 0$, then passing to the limits of \eqref{prce},
 we find that $x_1\to y_1$, $x_2\to y_2$, $x_3\to y_3$ and $\lambda\to \mu$, where $y_i$ and $\mu$ are non-negative numbers solving
 \begin{align}\label{yec0}
 \begin{split}
 0&=(4n+8)-2(y_1+y_2+y_3),\\
 0&=\mu y_1^2y_2y_3+2(y_1^2-(y_2-y_3)^2),\\
 0&=\mu y_2^2y_1y_3+2(y_2^2-(y_1-y_3)^2),\\
 0&=\mu y_3^2y_1y_2+2(y_3^2-(y_1-y_2)^2).
 \end{split}
 \end{align}
 First, we claim that at least two of $y_1,y_2,y_3$ are identical. To see this, note that by
 taking differences of the last three equations of \eqref{yec0}, we find
\begin{align}\label{yec01}
\begin{split}
  0&=(y_1-y_2)\left(\mu y_1y_2y_3+4(y_1+y_2-y_3)\right),\\
  0&=(y_2-y_3)\left(\mu y_1y_2y_3+4(y_2+y_3-y_1)\right),\\
   0&=(y_3-y_1)\left(\mu y_1y_2y_3+4(y_1+y_3-y_2)\right).
  \end{split}
 \end{align}
If all of $y_1,y_2,y_3$ are distinct, then at least one of $y_1,y_2,y_3$ is positive, and \eqref{yec01} implies that
\begin{align}\label{yec02}
\begin{split}
  0&=\left(\mu y_1y_2y_3+4(y_1+y_2-y_3)\right),\\
  0&=\left(\mu y_1y_2y_3+4(y_2+y_3-y_1)\right),\\
   0&=\left(\mu y_1y_2y_3+4(y_1+y_3-y_2)\right).
  \end{split}
 \end{align}
 By adding these three equations up, we find that the non-negative numbers $y_1,y_2,y_3$ and $\mu$ satisfy $3\mu y_1y_2y_3+4(y_1+y_2+y_3)=0$, which is a contradiction since
 one of $y_1,y_2,y_3$ is positive.

We now know that at least two of $y_1,y_2,y_3$ are identical, so we assume without loss of generality that $y_2=y_3$.
Then since $\mu,y_1,y_2$ and $y_3$ are all non-negative, the second equation of \eqref{yec0} implies that $y_1=0$, and the first equation implies that $y_2=y_3>0$.
Now by dividing the third and fourth equations of \eqref{prce} by $x_1x_3$ and $x_1x_2$ respectively, we see that
 \begin{align*}
 cT_2&=\lambda x_2^2+2\Big(\frac{x_2^2}{x_1x_3}-\frac{x_1}{x_3}-\frac{x_3}{x_1}+2\Big),\\
 cT_3&=\lambda x_3^2+2\Big(\frac{x_3^2}{x_1x_2}-\frac{x_1}{x_2}-\frac{x_2}{x_1}+2\Big).
 \end{align*}
 Since $x_1\to y_1=0$, $c\to 0$, $\lambda\ge0$ and $x_2,x_3\to y_2=y_3> 0$, we deduce that eventually, both $\frac{x_2^2}{x_1x_3}-\frac{x_3}{x_1}$
 and $\frac{x_3^2}{x_1x_2}-\frac{x_2}{x_1}$ must be negative. Thus $x_2^2-x_3^2<0$ and $x_3^2-x_2^2<0$, which is a contradiction.
\vspace{5pt}

 \textbf{Second} \textbf{Case}. Now $c$ converges to some positive number,
 but at least one of $x_1,x_2,x_3$ is converging to $0$.
 To start, assume that $x_1\to 0$. The third equation of \eqref{prce} implies that $x_2-x_3\to 0$. The second equation of \eqref{prce} then implies that
 $x_2x_3c\to 0$. Since $c$ does not converge to $0$, we must have both $x_2$ and $x_3$ converging to $0$ as well as $x_1$.
 If instead of assuming $x_1\to 0$ we assume that $x_2\to 0$ or
  $x_3\to 0$, we would again conclude that all three of $x_1,x_2,x_3$ are converging to $0$.

Since $x_1,x_2,x_3\to 0$, the first equation of \eqref{prce} implies that $c\to 4n+8$.
Rewrite the second, third and fourth equations as
\begin{align}\label{PRQE}
\begin{split}
 cT_1=\lambda x_1^2+2z_2z_3,\qquad
 cT_2=\lambda x_2^2+2z_1z_3,\qquad
 cT_3=\lambda x_3^2+2z_1z_2,
\end{split}
\end{align}
where $z_1=\frac{x_2+x_3-x_1}{x_1}$, $z_2=\frac{x_1+x_3-x_2}{x_2}$, $z_3=\frac{x_1+x_2-x_3}{x_3}$, and we can assume that these numbers
change monotonically as well. For each $i=1,2,3$, $\lambda x_i^2\to 0$ and $cT_i\to (4n+8)T_i>0$, so \eqref{PRQE} implies that
all three of $z_1,z_2,z_3$ are bounded, from which we deduce that $\frac{x_i}{x_j}$ is bounded for each $i$ and $j$. Now write \eqref{PRQE} as
\begin{align}\label{prce''}
\begin{split}
 cT_1&=\lambda x_1^2+\frac{2((dx_1)^2-(dx_2-dx_3)^2)}{dx_2dx_3},\\
cT_2&=\lambda x_2^2+\frac{2((dx_2)^2-(dx_1-dx_3)^2)}{dx_1dx_3},\\
 cT_3&=\lambda x_3^2+\frac{2((dx_3)^2-(dx_1-dx_2)^2)}{dx_2dx_1},
\end{split}
\end{align}
where $d=1/x_1$ (again dropping reference to the superscript).
By taking a subsequence, we can assume that $dx_2$ and $dx_3$ are monotone. Since $\frac{x_i}{x_j}$ is a bounded sequence for each $i$ and $j$,
we know that $d x_2$ and $dx_3$ converge to some positive numbers.
Taking limits and using the fact that $c\to 4n+8$, we see that the numbers $dx_i$  converge to a solution of \eqref{su2}
with $c=4n+8$. However, this contradicts Lemma~\ref{bound}.
\end{proof}

In order to solve \eqref{prce} we rewrite the equations in terms of smooth functions $f_\lambda\colon  (0,\infty)^4\to \mathbb{R}^4$, $\lambda\in [0,4n]$, such that $$f_\lambda(x_1,x_2,x_3,c)=(4n+8,T_1,T_2,T_3):=y$$ is equivalent to \eqref{prce}. We want to show that $f_\lambda^{-1}(y)$ is nonempty for all $\lambda$, which implies our theorem when $\lambda=4n$. We first show that this holds when $\lambda=0$. Recall that we can solve the last 3 equations in \eqref{prce}, which coincide with equations \eqref{su2} when $\lambda=0$, and that the solution is unique up to scaling. Lemma \ref{bound} implies that under the assumption $1/T_i<2n+4$, $i=1,2,3$,  we can choose this scaling so that the first equation  in \eqref{prce} is satisfied as well. Thus $f_0^{-1}(y)$ consists of a single point  $p\in \Omega$. Since the solution depends differentiably on $T_i$, it follows that $f^{-1}$ is differentiable near $p$ and hence $\det(Df_0)_p\ne 0$, which implies that $deg(f_0|_{\Omega},y)=\pm 1$. Notice that by Lemma \ref{APEPRCE} $\bar\Omega$ is compact, and $y\notin f_\lambda(\partial\Omega)$. Thus by  the homotopy invariance of the Brouwer degree, it follows that $deg(f_\lambda|_{\Omega},y)\ne 0$ for all $\lambda$. This finishes the proof.
\end{proof}

The condition in Theorem~\ref{thm_PRC_4n3} is not necessary.  In fact, its proof shows that one has the following stronger statement:

\begin{theorem}\label{cond_c<4n+8} There exists an $\Sp(n+1)$-invariant metric $g$  such that $\Ric g=\kappa T$ for some $\kappa>0$ if one has
\begin{align}\label{eq_suff_cond 2}
c=c(T_1/b,T_2/b,T_3/b)<4n+8,
\end{align}
where $c(x_1,x_2,x_3)$ is defined in terms of the solution to the system~\eqref{su2}.
\end{theorem}

It is conceivable that this is necessary and sufficient for the solvability of $\Ric g=\kappa T$.
It is easy to find an explicit expression for the function~$c$. In fact,
$$
c(x_1,x_2,x_3)=8\frac{x_1Z^2-x_3 }{x_1^2Z^2-x_3^2 },
$$
where $Z$ is a solution of the cubic equation
$$
x_1^2(x_2-x_3)Z^3+ x_1x_3(2x_1-x_2-x_3)Z^2+x_1x_3(2x_3-x_1-x_2)Z+x_3^2(x_2-x_1)=0.
$$

Among the four-parameter family of metrics $g_{(x_1,x_2,x_3,s)}$ there exists a special subclass with $x_2=x_3=x$. These metrics are invariant under the larger isometry group $\Sp(n+1)\U(1)$, where $\U(1)$ acts by rotation in the plane corresponding to the $x_2,x_3$ variables. If we restrict our attention to the $\Sp(n+1)\U(1)$-invariant case, we can obtain a simple formula for the function $c$ appearing in~\eqref{cond_c<4n+8}. In this case, we have the following improved version of Theorem~\ref{thm_PRC_4n3}:

\begin{theorem}\label{thm_SpU_PRC}
Assume that the metric $T$ on $\mathbb S^{4n+3}$ given by~(\ref{T_4comp}) is $\Sp(n+1)\U(1)$-invariant, i.e., $T_2=T_3$. If
\begin{align*}
bT_1+4bT_2-b\sqrt{T_1^2+8T_1T_2}<(4n+8)T_2^2,
\end{align*}
then there exists an $\Sp(n+1)\U(1)$-invariant metric $g$ such that $\Ric g=\kappa T$ for some $\kappa>0$.
\end{theorem}

\begin{proof}
Straightforward verification shows that
$$
(x_1,x_2,x_3,c)=\Big(\frac1{2T_2}\Big(-T_1+\sqrt{T_1^2+8T_1T_2}\Big)x,x,x,\frac b{T_2^2}\Big(T_1+4T_2-\sqrt{T_1^2+8T_1T_2}\Big)\Big)
$$
is a solution to~\eqref{su2} with $T_i$ replaced by $T_i/b$ for every $x>0$. The claim follows from this observation and Theorem~\ref{cond_c<4n+8}.
\end{proof}

\begin{remark}\label{limits}
This shows that the condition \eqref{eq_suff_cond 2} is {substantially} weaker than~\eqref{eq_suff_cond}. For example, suppose that $b=1$ and $T_2=T_3>\frac1{n+2}$. In this case,
\begin{align*}
c(T_1/b,T_2/b,T_3/b)=\frac b{T_2^2}\Big(T_1+4T_2-\sqrt{T_1^2+8T_1T_2}\Big)\le\frac4{T_2}<4n+8,
\end{align*}
which means~\eqref{cond_c<4n+8} holds for all $T_1>0$. However,~\eqref{eq_suff_cond} only holds for $T_1>\frac1{2n+4}$.
\end{remark}

\begin{remark}
If an $\Sp(n+1)\U(1)$-invariant metric $g_{(x_1,x,x,s)}$ admits an ancient iteration, then the following necessary condition holds: $x_1\le x\le s$. The proof of this fact requires careful analysis of monotone quantities associated  with the iteration. 
\end{remark}

\begin{remark}
For  homogeneous metrics the Ricci flow can be interpreted as the gradient flow of the negative scalar curvature restricted to the space of metrics of volume 1. The round metric is a local maximum of this functional, and the second Einstein metric a saddle point. It has two negative eigenvalues and one positive eigenvalue whose eigenvector is $x_1=x_2=x_3$. Thus, at least near the second Einstein metric, there exist precisely two ancient solutions, already contained in the family $g_{t,s}$; see~\cite{MB14}.
\end{remark}

\section{Ricci iteration near the round metric}\label{nearround}

The general behavior of Ricci iterations among the family $g_{(x_1,x_2,x_3,s)}$ seems to be difficult to understand. But one can descibe their behavior near a round metric.

The set of homogeneous metrics on $\Sph^{4n+3}$ has a natural topology since it is determined by the inner products on the tangent space at one point, which in turn form  an open cone in the vector space of symmetric bilinear forms on the tangent space. We let $h_0$ be the round metric $g_{(1,1,1,1)}$.

\begin{theorem}
There exists a neighborhood $\mathcal O$ of the round metric $h_0$ in the space of homogeneous metrics on $\Sph^{4n+3}$  such that the following  holds:
if $g_0\in\mathcal O$, then there exists a Ricci iteration that starts with $cg_0$ for some $c>0$ and converges to a metric of constant curvature.
\end{theorem}

\begin{proof}
Choose $\Omega$ to be some open subset of
$$(0,\infty)^3\setminus\{(x_1,x_2,x_3)\in(0,\infty)^3:x_1+x_2+x_3\ne2n+4\},$$
and let $f=(f_1,f_2,f_3):\Omega\to \mathbb{R}^3$ be given by
\begin{align*}
f_i(x_1,x_2,x_3)&=\frac{4nx_i^2+2\frac{(x_i+x_{i+1}-x_{i+2})(x_i+x_{i+2}-x_{i+1})}{x_{i+1}x_{i+2}}}{(4n+8)-2(x_1+x_2+x_3)}.
\end{align*}
Equations \eqref{prce} imply that if $y_i=f(x_1,x_2,x_3)$, then $\Ric g_{(x_1,x_2,x_3,1)}=c T_{(y_1,y_2,y_3,1)}$, where we have $c=(4n+8)-2(x_1+x_2+x_3)$.
 Therefore, if we have a sequence
$$\big\{x^{(k)}=\big(x^{(k)}_1,x^{(k)}_2,x^{(k)}_3\big)\big\}_{k=1}^\infty\in (0,\infty)^3$$
such that
$x^{(k)}=f(x^{(k+1)})$ for $k\ge 1$, then there exist some $c^{(k)}>0$ such that the metrics
 \begin{equation}\label{iterate}
g_k=g_{(c^{(k)}x_1^{(k)},c^{(k)}x_2^{(k)},c^{(k)}x_3^{(k)},c^{(k)})}
\end{equation} form a Ricci iteration.

Next, observe that the derivative of $f$ at $(1,1,1)$ is $2+\frac{1}{2n+1}$
times the identity. Therefore, there exist neighborhoods $\Omega,\Omega'$ of $(1,1,1)$ such that
 $f:\Omega\to \Omega'$ is a diffeomorphism. Since $\left|Df^{-1}\right| <1$ at $(1,1,1)$, we can assume, by making $\Omega'$ smaller,  that $f^{-1}(\Omega')\subset \Omega'$ and
$\left|Df^{-1}\right|\le a<1$ on all of $\Omega'$. Thus $f^{-1}$ is a contraction and hence $x^{(k+1)}=f^{-1}(x^{(k)})$ is a sequence that converges to $(1,1,1)$.

Let $\mathcal O$ be any neighborhood of $h_0$ in the space of homogeneous metrics such that every metric in $\mathcal O$ can be transformed by a gauge transformation  into a metric of the form $g_{(x_1,x_2,x_3,s)}$ with $(x_1/s,x_2/s,x_3/s)\in\Omega'$.
Fix a metric $g_0\in\mathcal O$.
Without loss of generality, assume $g_0$ is diagonal. Then  \eqref{iterate} defines a sequence of metrics with $\Ric g_{k+1}=g_k$ for $k\ge 1$
and $\Ric g_1=cg_0$ for some $c>0$. Furthermore, $\big(x_1^{(k)},x_2^{(k)},x_3^{(k)}\big)$ converges to $(1,1,1)$ and since $c^{(k)}=4n+8-2\big(x_1^{(k+1)}+x_2^{(k+1)}+x_3^{(k+1)}\big)$, it follows that $g_k$ converges to a metric of constant curvature.
\end{proof}

\begin{remark}
The theorem shows that the round metric is stable for the Ricci iteration. In fact, it is also stable for the Ricci flow since it is a local maximum of the scalar curvature functional on the set of metrics of volume~1.
\end{remark}

\appendix
\section{}

If a tensor $T$ on $\mathbb S^{4n+3}$ is invariant under $\Sp(n+1)\Sp(1)$, it is natural to ask whether every $\Sp(n+1)$-invariant solution to $\Ric g=cT$ must also be $\Sp(n+1)\Sp(1)$-invariant. The following theorem shows that this is indeed the case. As we explained in Section~\ref{sec_4n3}, every $\Sp(n+1)$-invariant metric on $\mathbb S^{4n+3}$ can be written as $g_{(x_1,x_2,x_3,s)}$ for some $x_1,x_2,x_3,s>0$. By definition, $g_{(x_1,x_2,x_3,s)}$ is obtained by modifying a round metric $\hat g$ with a left-invariant metric on the 3-dimensional fibers of the Hopf fibration and scaling in the horizontal direction. If $x_1=x_2=x_3=t$, then $g_{(x_1,x_2,x_3,s)}$ coincides with the $\Sp(n+1)\Sp(1)$-invariant metric $g_{t,s}$ obtained by scaling $\hat g$ by $t$ in the direction of the fibers and by $s$ in the perpendicular direction.

\begin{theorem}\label{sp1}
Let $g$ be an $\Sp(n+1)$-invariant metric on $\mathbb S^{4n+3}$ with $\Sp(n+1)\Sp(1)$-invariant Ricci curvature. Then, up to isometry, $g=g_{t,s}$ for some $t,s$.
\end{theorem}

\begin{proof}
Assume $g=g_{(x_1,x_2,x_3,s)}$. By the scale-invariance of the Ricci curvature, it suffices to consider the case where $s=1$. Our goal is to show that $x_1=x_2=x_3$.

Since $\Ric g$ is $\Sp(n+1)\Sp(1)$-invariant, it can be obtained by multiplying $\hat g$ by some $a\in\mathbb R$ in the direction of the fibers and by some $b\in\mathbb R$ in the perpendicular direction. Computing $\Ric g$ as in Section~\ref{sec_4n3}, we find
\begin{align*}
b&=4n+8-2(x_1+x_2+x_3),\notag \\
a&=4nx_i^2+2\frac{x_i^2-(x_{i+1}-x_{i+2})^2}{x_{i+1}x_{i+2}},\qquad i=1,2,3.
\end{align*}
Let us multiply the second line by $x_{i+1}x_{i+2}$ and take differences of the equations for different~$i$. We see that
\begin{equation}\label{takedifference}
\begin{aligned}
0&=(x_1-x_2)(4nx_1x_2x_3+4(x_1+x_2-x_3)+ax_3), \\
0&=(x_2-x_3)(4nx_1x_2x_3+4(x_2+x_3-x_1)+ax_1).
\end{aligned}\tag{A.1}
\end{equation}
First assume that $x_1,x_2,x_3$ are all distinct. Then
\begin{align*}
0&=4nx_1x_2x_3+4(x_1+x_2-x_3)+ax_3,\\
0&=4nx_1x_2x_3+4(x_2+x_3-x_1)+ax_1.
\end{align*}
Taking the difference of these equalities, we see that $8(x_1-x_3)+a(x_3-x_1)=0$ and $a=8$,
so
$$4nx_1x_2x_3+4(x_1+x_2-x_3)+ax_3=4nx_1x_2x_3+4(x_1+x_2+x_3)>0,$$
a contradiction. Thus at least two of $x_1,x_2,x_3$ are identical, and in this case it is a simple matter to conclude from~\eqref{takedifference} that all three must in fact be identical.
\end{proof}

\end{document}